\documentclass[reqno]{amsart}   
\usepackage{amssymb,amsthm}
\theoremstyle{plain}
\usepackage{tikz}
\usepackage{enumerate,enumitem}
\usepackage{thmtools, thm-restate}
\declaretheorem{theorem}

\numberwithin{theorem}{section}
\newtheorem{lemma}[theorem]{Lemma}

\newtheorem{conj}{Conjecture}

\newcommand{\N}{\mathbb{N}}

\definecolor{blue}{RGB}{0,83,159}
\definecolor{black}{RGB}{0,0,0}
\definecolor{white}{RGB}{255,255,255}
\definecolor{lightblue}{RGB}{142,186,226}
\definecolor{grey}{RGB}{51,51,51}
\definecolor{lightgrey}{RGB}{204,204,204}
\definecolor{superlightgrey}{RGB}{247,247,247}
\definecolor{petrol}{RGB}{0,97,101}
\definecolor{teal}{RGB}{0,152,161}
\definecolor{maygreen}{RGB}{189,205,0}
\definecolor{green}{RGB}{87,171,39}
\definecolor{yellow}{RGB}{255,237,0}
\definecolor{orange}{RGB}{246,168,0}
\definecolor{magenta}{RGB}{227,0,102}
\definecolor{red}{RGB}{204,7,30}
\definecolor{bordeaux}{RGB}{161,16,53}
\definecolor{violet}{RGB}{97,33,88}
\definecolor{purple}{RGB}{122,111,172}
\definecolor{darkgreen}{rgb}{0.2,0.7,0.2}
\definecolor{halfgray}{gray}{0.65}

\definecolor{grey}{RGB}{51,51,51}
\definecolor{lightgrey}{RGB}{204,204,204}

\begin{document}
\title[A counterexample to Montgomery's conjecture]{A counterexample to Montgomery's conjecture on dynamic colourings of regular graphs}   
\author[Bowler, Erde,  Lehner, Merker, Pitz \& Stavropoulos]{Nathan Bowler, Joshua Erde, Florian Lehner, Martin Merker, Max Pitz, Konstantinos Stavropoulos}
\thanks{The second author was supported by the Alexander von Humboldt Foundation.}
\address{University of Hamburg, Department of Mathematics, Bundesstra{\ss}e 55 (Geomatikum), 20146 Hamburg, Germany}
\email{nathan.bowler@uni-hamburg.de, joshua.erde@uni-hamburg.de,\newline florian.lehner@uni-hamburg.de, martin.merker@uni-hamburg.de, max.pitz@uni-hamburg.de, konstantinos.stavropoulos@uni-hamburg.de}

\begin{abstract}
A \emph{dynamic colouring} of a graph is a proper colouring in which no neighbourhood of a non-leaf vertex is monochromatic. The \emph{dynamic colouring number} $\chi_2(G)$ of a graph $G$ is the least number of colours needed for a dynamic colouring of $G$. 

Montgomery conjectured that $\chi_2(G) \leq \chi(G) + 2$ for all regular graphs $G$, which would significantly improve the best current upper bound $\chi_2(G) \leq 2\chi(G)$. In this note, however, we show that this last upper bound is sharp by constructing, for every integer $n \geq 2$, a regular graph $G$ with $\chi(G) = n$ but $\chi_2(G) = 2n$. In particular, this disproves Montgomery's conjecture.
\end{abstract}

\maketitle
\section{Introduction}
A \emph{proper $n$-colouring} of a graph $G$ is a map $c \colon V(G) \rightarrow [n]$ such that there is no edge $uv$ with $c(u)=c(v)$. An \emph{$r$-dynamic $n$-colouring} is a proper $n$-colouring $c$ of $G$ such that for each vertex $v \in V(G)$ at least $\min \{ r,d(v) \}$ colours are used on $N(v)$. The \emph{$r$-dynamic chromatic number} $\chi_r(G)$, introduced by Montgomery \cite{M01}, is the smallest $n$ such that $G$ has an $r$-dynamic $n$-colouring. 
Problems on $r$-dynamic colourings of graphs are an active area of research in graph theory, with many recent papers investigating properties of $r$-dynamic colourings, e.g.\ \cite{AGJ09,CFLSS12,KMW15,KLO16,LMP03}. 

Of particular interest is the 2-dynamic chromatic number, which is often simply called the \emph{dynamic chromatic number}.
Note that $\chi_2(G) - \chi(G)$ can be arbitrarily large, which can be seen by considering subdivisions of complete graphs. However, in the case of regular graphs, Montgomery \cite{M01} conjectured the following:

\begin{conj}[Montgomery's Conjecture]\label{c:mont}
If $G$ is a regular graph, then $\chi_2(G) \leq \chi(G) + 2$.
\end{conj}

Ahadi, Akbari, Deghan, and Ghanbari \cite{AADG12} conjectured further that if $\chi(G) \geq 4$ then $\chi(G) = \chi_2(G)$. However Alishahi \cite{A12} disproved the stronger conjecture by constructing graphs $G$ with $\chi(G)=n$ such that $\chi_2(G) \geq \chi(G) +1$ for each $n \geq 2$. 

There are several results showing that Conjecture \ref{c:mont} holds for certain classes of regular graphs. For example, Montgomery \cite{M01} showed that Conjecture~\ref{c:mont} holds for claw-free graphs, and Akbari, Ghanbari and Jahanbekam \cite{AGJ10} showed that Conjecture~\ref{c:mont} holds for $d$-regular bipartite graphs with $d \geq 4$. Alishahi \cite{A12} verified Conjecture~\ref{c:mont} if $G$ has diameter at most $2$ and chromatic number at least $4$.

There are also results bounding the dynamic chromatic number of regular graphs in terms of other invariants of the graph. For example, Dehghan, and Ahadi \cite{DA12} showed that if $G$ is regular, then $\chi_2(G) \leq \chi(G) + 2\log{\alpha(G)} +3$, where $\alpha$ is the independence number of the graph. Alishahi \cite{A11} showed that if $G$ is $d$-regular, then $\chi_2(G) \leq \chi(G) + 14.06\log{d} + 1$, and Taherkhani \cite{T16} improved this bound to  $\chi_2(G) \leq \chi(G) + 5.437\log{d} + 2$. The best general bound for $\chi_2(G)$ depending only on $\chi(G)$ was given by Alishahi \cite{A11} who showed that $\chi_2(G) \leq 2 \chi(G)$ for all regular graphs $G$. 

in this paper we show that for every value of $\chi(G)$ there exists a graph for which this inequality is sharp, giving an infinite family of counterexamples to Montgomery's Conjecture. 

\begin{theorem}\label{t:main1}
For every natural number $n\geq 2$, there exists a regular graph $G_n$ with $\chi(G_n) = n$ and $\chi_2(G_n)=2 \chi (G_n)$.
\end{theorem}

In general, the $r$-dynamic chromatic number of regular graphs cannot be bounded by $r\chi_r(G)$ as was shown by
Jahanbekam, Kim, O, and West \cite{JKSW16}. They showed that if $d=r$ then for infinitely many values of $r$ there exists a $d$-regular graph $G$ such that $\chi_r(G) \geq r^{1.377}\chi(G)$.
However, they also showed that $\chi_r(G) \leq r \chi(G)$ for $d\geq \left(3 + o(1) \right) r \log{r}$. 
Here we show that this bound is sharp even for arbitrarily large values of $d$.

\begin{theorem}\label{t:main2}
For all natural numbers $r,n,\delta \geq 2$, there exists a $d$-regular graph $G=G(r,n,\delta)$ with $d> \delta$, $\chi(G) = n$ and $\chi_r(G) = r \chi(G)$.
\end{theorem}

The dynamic chromatic number also has connections to the chromatic number of total dominating sets, which was observed by Alishahi \cite{A11}. We say a set $D \subset V(G)$ is a \emph{total dominating set} of $G$ if every vertex in $V(G)$ is adjacent to some vertex of $D$. Let us write 
\[
\gamma(G):= \min \{ \chi(G[D]) \,: \, D \text{ a total dominating set of } G \}.
\]
In the 1970s, Berge (unpublished, see \cite{CH76}) and independently Payan \cite{P74} conjectured that every regular graph contains two disjoint maximal independent sets. Payan \cite{P77,P78} showed that this is equivalent to $\gamma(G) = 2$ for every regular graph $G$ and disproved the conjecture by constructing regular graphs $G_p$ with $\gamma(G_p) > p$ for every $p\in \N$. His graphs have the property that $\chi(G_p)=4p+1$ and $\gamma(G_p)=p+1$. The graphs we construct here are new examples showing that $\gamma(G)$ is unbounded for regular graphs. Moreover, our graphs have the even stronger property that $\chi(G) = \gamma(G)$.

\begin{theorem}
For every natural number $n\geq 2$, there exists a regular graph $G_n$ with $\chi(G_n)=n=\gamma(G_n)$.
\end{theorem}

\section{Construction of the counterexamples}
Our aim is to build, for fixed natural numbers $r,n, \delta \geq 2$, a $d$-regular graph $G:=G(r,n,\delta)$ with $d> \delta$ and $\chi(G)=n$ but $\chi_r(G)=rn$ and $\gamma(G)=n$. The graph $G$ will consist of a disjoint union of many complete $n$-partite graphs, together with some supplementary vertices whose neighbourhoods are carefully chosen subsets of that disjoint union. 

More formally, letting $m := \max \{\binom{rn-1}{r-1}n^2, \delta\}$ and $N := \binom{m - 1}{n-1}$, we take a vertex $v_{i,j,k}$ for each $i \in [n]$, $j \in [N]$ and $k \in [m]$ and we add an edge from $v_{i,j,k}$ to $v_{i',j',k'}$ whenever $i \neq i'$ and $k =k'$. Additionally we take vertices $s_{i,X}$ for every $i \in [n]$ and every $X \subseteq [m]$ of size $n$, and we join $s_{i,X}$ to $v_{i',j,k}$ whenever $i = i'$ and $k \in X$ (see Fig \ref{fig:graph}). We call the graph with these vertices and edges $G$.

\begin{figure}
 \centering

\begin{tikzpicture}[
  dot/.style={draw,fill=black,circle,inner sep=0.6pt}
  ]
\foreach \k /\m in {1/1,2/2,3/k,4/m}{
\begin{scope}[xshift=80*\k-200]
\node at (90:0) {{\footnotesize $G_{\m}$}};
\foreach \n in {1,...,4}{
 \foreach \i in {1,2,3}{
    \node[dot] (\i\n\k) at (\i*360/3+70+8*\n:1.15) {};
  }  
  }
\foreach \n in {1,...,4}{
\foreach \l in {1,...,4}{
  \foreach \j in {2,3}{
   \draw[lightblue] (1\n\k) -- (\j\l\k);
    }
   \draw[lightblue] (2\n\k) -- (3\l\k);
}
}
\foreach \i in {1,2,3}{
\node at (120*\i-30:0.8) {{\tiny $Y_{\i,\m}$}};
}
\end{scope}
}

\node[fill=black,circle,inner sep=1pt,label=below:{\tiny{$s_{2,\{2,k,m\}}$}}] (s2-234) at (-3,-3) {};
\foreach \k in {2,3,4}{
 \foreach \n in {1,...,4}{
\draw[darkgreen] (s2-234) -- (1\n\k);
}
}

\node[fill=black,circle,inner sep=1pt,label=below:{\tiny{$s_{3,\{1,2,k\}}$}}] (s3-123) at (3,-3) {};
\foreach \k in {1,2,3}{
 \foreach \n in {1,...,4}{
\draw[darkgreen] (s3-123) -- (2\n\k);
}
}

\node[fill=black,circle,inner sep=1pt,label=above:{\tiny{$s_{1,\{1,2,m\}}$}}] (s5-12m) at (1,2) {};
\foreach \k in {1,2,4}{
 \foreach \n in {1,...,4}{
\draw[darkgreen] (s5-12m) -- (3\n\k);
}
}

\node[xshift=80] {{\footnotesize $\ldots$}};
\node {{\footnotesize $\ldots$}};

\end{tikzpicture}
\caption{$G$ is composed of $m$ disjoint complete $n$-partite graphs $(G_k)_{k\in [m]}$ each with parts $(Y_{i,k})_{i\in [n]}$ of size $N$, along with vertices $s_{i,X}$ for every $(i,X)\in [n]\times\binom{m}{n}$, where $s_{i,X}$ is connected to every vertex in $\bigcup_{k\in X}Y_{i,k}$.}
\label{fig:graph}
\end{figure}
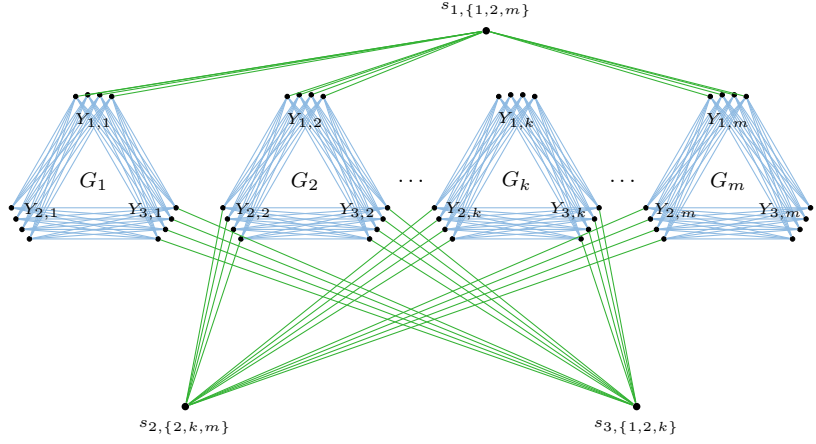

\begin{lemma}
$G$ is $(nN)$-regular.
\end{lemma}
\begin{proof}
It is clear that vertices of the form $s_{i,X}$ have degree $nN$. Now we consider the neighbours of a vertex of the form $v_{i,j,k}$. This vertex has $(n-1)N$ neighbours of the form $v_{i', j', k'}$, and it has $N$ neighbours of the form $s_{i,X}$ since there are $\binom{m-1}{n-1}$ subsets $X$ of $[m]$ of size $n$ containing $k$. So in total it has $nN$ neighbours.
\end{proof}

\begin{lemma}
$\chi(G)=n$.
\end{lemma}
\begin{proof}
The chromatic number is at least $n$, since the subgraph on the $v_{i,1,1}$ with $i \in [n]$ is complete. But $G$ can be properly coloured with $n$ colours by assigning each vertex $v_{i,j,k}$ the colour $i$ and each vertex $s_{i,X}$ a colour other than $i$.
\end{proof}

\begin{lemma}
$\chi_r(G)=rn$.
\end{lemma}
\begin{proof}
We can $r$-dynamically colour $G$ with $rn$ colours by assigning each vertex $v_{i,j,k}$ with $j < r$ the colour $nj + i$, each vertex $v_{i,j,k}$ with $j \geq r$ the colour $i$, and each vertex $s_{i,X}$ a colour which is not congruent to $i$ modulo $n$.

Now consider a proper colouring of $G$ with at most $rn-1$ colours. We shall show that this colouring is not $r$-dynamic. For each $k \in [m]$ and each colour $c$, there is at most one $i \in [n]$ such that for some $j$ the vertex $v_{i, j, k}$ has colour $c$. Since there are only $rn-1$ colours, there must be some $i_k \in [n]$ and some set $C_k$ of colours of size at most $r-1$ such that all $v_{i_k,j,k}$ are coloured with a colour from $C_k$. By the pigeonhole principle, since $m>(n-1)n\binom{rn-1}{r-1}$, there is a set $X \subseteq [m]$ of size $n$ such that all $i_k$ with $k \in X$ take some common value $i$ and all $C_k$ with $k \in X$ take some common value $C$. But then all neighbours of $s_{i,X}$ are coloured with colours from $C$, and so the colouring is not $r$-dynamic.
\end{proof}

\begin{lemma}
$\gamma(G)=n$.
\end{lemma}
\begin{proof}
Consider a set $S\subseteq V(G)$ with $\chi(G[S])\leq n-1$. For each $k \in [m]$ there exists at least one $i_k \in [n]$ such that $S$ contains none of the $v_{i_k,j,k}$. Since $m>(n-1)n$, by the pigeonhole principle, there exists a set $X \subseteq [m]$ of size $n$ such that all $i_k$ with $k \in X$ take some common value $i$. Now none of the neighbours of $s_{i,X}$ are in $S$, so $S$ is not a total dominating set.
\end{proof}

\bibliographystyle{abbrv}
\bibliography{dynamic}

\end{document}